\documentclass[english]{amsart}

\usepackage{amsmath}
\usepackage{amssymb}
\usepackage{amsthm}
\usepackage{amscd}
\usepackage{babel}
\usepackage[latin1]{inputenc}
\usepackage{graphicx}

\newtheorem{teor}{Theorem}
\newtheorem*{teorpre}{Theorem}
\newtheorem{cor}{Corollary}
\newtheorem{prop}{Proposition}
\newtheorem{con}{Conjecture}
\newtheorem{lem}{Lemma}

\theoremstyle{definition}

\newtheorem*{rem}{Remark}

\renewcommand{\subjclassname}{AMS \textup{2010} Mathematics Subject
Classification\ }

\author{P. Fortuny Ayuso}
\address{Departamento de Matemáticas, Universidad de Oviedo\\ Avda. Calvo Sotelo s/n, 33007 Oviedo, Spain}
\email{fortunypedro@uniovi.es}

\author{Jos\'{e} Mar\'{i}a Grau}
\address{Departamento de Matemáticas, Universidad de Oviedo\\ Avda. Calvo Sotelo s/n, 33007 Oviedo, Spain}
\email{grau@uniovi.es}

\author{Antonio M. Oller-Marc\'{e}n}
\address{Centro Universitario de la Defensa de Zaragoza\\ Ctra. Huesca s/n, 50090 Zaragoza, Spain} \email{oller@unizar.es}

\title{On the congruence $\displaystyle{\sum_{z\in\mathbb{Z}_n[\mathfrak{i}]} z^k }$}

\title{A von Staudt-type formula for 
$\displaystyle{\sum_{z\in\mathbb{Z}_n[i]} z^k }$}

\begin{document}

\begin{abstract} In this paper we study the sum of powers in the
Gaussian integers $\mathbf{G}_k(n):=\sum_{a,b \in [1,n]} (a+b
i)^k$. We give an explicit formula for $\mathbf{G}_k(n) \pmod n$
in terms of the prime numbers $p \equiv 3 \pmod 4$ with $p \mid \mid  
n$ and $p-1 \mid k$, similar to the well known one due to von Staudt
for $\sum_{i=1}^n i^k \pmod n$. We apply this formula to study the set
of integers $n$ which divide $\mathbf{G}_n(n)$ and compute its asymptotic
density with six exact digits:
$0.971000\ldots$.  
\end{abstract}

\maketitle \subjclassname{11B99, 11A99, 11A07}

\keywords{Keywords: Power sum, Erdös-Moser equation, Asymptotic
density}

\section{Introduction}

The sum of powers of integers of the form
$$S_k(n):=1^k+2^k+3^k+\cdots+n^k$$
is a well-studied problem in arithmetic
(see e.g., \cite{int1} and \cite{int2}).
Finding formulas for these sums has interested mathematicians for more than 300 years
since the time of James Bernoulli (1665-1705). If we call
$B_i$ and $B_i(x)$ the $i$-th Bernoulli number and Bernoulli
polynomial, respectively, then (see, e.g., \cite{BEA})
\begin{equation}\label{int}S_k(m)=
\frac{B_{k+1}(m+1)-B_{k+1}}{k+1}.\end{equation}

The sum of powers modulo $n$ was studied by von Staudt in 1840 in
\cite{VON}, where he gave the following result for even $k$:

\begin{teor}
  Let $k,n\geq 1$ be integers with $k$ even, then,
  $$
  S_k(n)\equiv \displaystyle{-\sum_{\substack{p \mid n \\ p-1 \mid k
      }}\frac{n}{p_i} \pmod{n}}.
  $$
\end{teor}

L. Carlitz \cite{CAR} considered the case $k$ odd and
claimed that $n \mid S_k(n)$ in that case. P. Moree \cite{MOR2}
pointed out that this is false, but that $S_k(n) = r n/2$
for integer $r$.
The following lemma from a preprint of \cite{KEL1} gives the
parity of r:

\begin{lem}
  Let $k >2$ be odd. There is an integer $r$ such that $S_k(n) = rn/2$.
  If $n \equiv 2 \pmod 4$ then $r$ is odd, otherwise it is even.
\end{lem}

\newpage
On the other hand, in \cite{GMO} the pairs $(k,n)$ with
$k,n\geq 1$ such that $n\mid S_k(n)$ were characterized. In
particular:

\begin{teor}
\label{TEO:GMO} Let $k,n\geq 1$ be integers. Then, $n\mid S_k(n)$ if
and only if one of the following holds:
\begin{itemize}
\item[i)] $n$ is odd and $p-1\nmid k$ for every prime divisor $p$ of
$n$.
\item[ii)] $n$ is a multiple of $4$ and $k>1$ is odd.
\end{itemize}
\end{teor}

Much research has been carried out regarding divisibility
properties of power sums (see for instance \cite{LEN,SON2,ARX,SON}).

In this work, we deal with power sums of Gaussian integers, an extension that
has not been considered yet. Instead of the sum of the $k$-th powers
of the first $n$ positive integers, we are concerned with the sum
of the $k$-th powers of all Gaussian integers in the $n \times n$ base
square of the first quadrant. Namely, this paper deals with power sums
of the form:

$$
\mathbf{G}_k(n):=
\sum_{1\leq a,b\leq
  n}(a+bi)^k.
$$

Table \ref{TAB} lists the values of $\mathbf{G}_k(n) \pmod{n}$ for
$1\leq k,n\leq 24$.

\begin{table}
\label{TAB}
\caption{$\mathbf{G}_k(n)\pmod{n}$ for $1\leq k,n\leq 24$; whit $\epsilon:=(1+i)$.}

\scriptsize
\begin{tabular}{|p{0.5cm}|p{0.01cm}|p{0.01cm}|p{0.01cm}|p{0.01cm}|p{0.03cm}|p{0.2cm}|p{0.03cm}|p{0.03cm}|p{0.03cm}|p{0.2cm}|p{0.2cm}|p{0.14cm}|p{0.14cm}|
p{0.2cm}|p{0.14cm}|p{0.14cm}|p{0.14cm}|p{0.2cm}|p{0.14cm}|p{0.14cm}|p{0.14cm}|p{0.3cm}|p{0.14cm}|p{0.14cm}|}
\hline 

$k \diagdown n$ &1 & 2 & 3 & 4 &5 & 6 & 7 & 8 & 9& 10 & 11 & 12 & 13 &
14 & 15 & 16 & 17 & 18 & 19 & 20 & 21 & 22 & 23 & 24 \\
\hline 1 &
\textbf{0} & 0 & 0 & 0 & 0 & 0 & 0 & 0 & 0 & 0 & 0 & 0 & 0 & 0 & 0 & 0
& 0 & 0 & 0 & 0 & 0 & 0 & 0 & 0 \\
\hline 2 & 0 & \textbf{0} & 0 & 0 &
0 & 0 & 0 & 0 & 0 & 0 & 0 & 0 & 0 & 0 & 0 & 0 & 0 & 0 & 0 & 0 & 0 & 0
& 0 & 0 \\
\hline 3 & 0 & $\epsilon$ & \textbf{0} & 0 & 0 &
3$\epsilon$ & 0 & 0 & 0 & 5$\epsilon$& 0 & 0 & 0 &7$\epsilon$& 0 & 0 &
0 & 9$\epsilon$ & 0 & 0 & 0 & 11$\epsilon$ & 0 & 0 \\
\hline 4 & 0 & 0
& 0 & \textbf{0} & 0 & 0 & 0 & 0 & 0 & 0 & 0 & 0 & 0 & 0 & 0 & 0 & 0 &
0 & 0 & 0 & 0 & 0 & 0 & 0 \\
\hline 5 & 0 & $\epsilon$ & 0 & 0 &
\textbf{0} & 3$\epsilon$ & 0 & 0 & 0 & 5$\epsilon$ & 0 & 0 & 0
&7$\epsilon$& 0 & 0 & 0 & 9$\epsilon$ & 0 & 0 & 0 & 11$\epsilon$ & 0 &
0 \\
\hline 6 & 0 & 0 & 0 & 0 & 0 & \textbf{0} & 0 & 0 & 0 & 0 & 0 & 0
& 0 & 0 & 0 & 0 & 0 & 0 & 0 & 0 & 0 & 0 & 0 & 0 \\
\hline 7& 0 &
$\epsilon$ & 0 & 0 & 0 & 3$\epsilon$ & \textbf{0} & 0 & 0 &
5$\epsilon$ & 0 & 0 & 0 &7$\epsilon$& 0 & 0 & 0 & 9$\epsilon$ & 0 & 0
& 0 & 11$\epsilon$ & 0 & 0 \\
\hline 8 & 0 & 0 & 2 & 0 & 0 & 2 & 0 &
\textbf{0} & 0 & 0 & 0 & 8 & 0 & 0 & 5 & 0 & 0 & 0 & 0 & 0 & 14 & 0 &
0 & 8 \\
\hline 9 & 0 & $\epsilon$ & 0 & 0 & 0 & 3$\epsilon$ & 0 & 0 &
\textbf{0} & 5$\epsilon$ & 0 & 0 & 0 &7$\epsilon$& 0 & 0 & 0 &
9$\epsilon$ & 0 & 0 & 0 & 11$\epsilon$ & 0 & 0 \\
\hline 10 & 0 & 0 &
0 & 0 & 0 & 0 & 0 & 0 & 0 & \textbf{0} & 0 & 0 & 0 & 0 & 0 & 0 & 0 & 0
& 0 & 0 & 0 & 0 & 0 & 0 \\
\hline 11 & 0 & $\epsilon$ & 0 & 0 & 0 &
3$\epsilon$ & 0 & 0 & 0 & 5$\epsilon$ & \textbf{0} & 0 & 0
&7$\epsilon$& 0 & 0 & 0 & 9$\epsilon$ & 0 & 0 & 0 & 11$\epsilon$ & 0 &
0 \\
\hline 12 & 0 & 0 & 0 & 0 & 0 & 0 & 0 & 0 & 0 & 0 & 0 &
\textbf{0} & 0 & 0 & 0 & 0 & 0 & 0 & 0 & 0 & 0 & 0 & 0 & 0 \\
\hline
13 & 0 & $\epsilon$ & 0 & 0 & 0 & 3$\epsilon$ & 0 & 0 & 0 &
5$\epsilon$ & 0 & 0 & \textbf{0} &7$\epsilon$& 0 & 0 & 0 & 9$\epsilon$
& 0 & 0 & 0 & 11$\epsilon$ & 0 & 0 \\
\hline 14 & 0 & 0 & 0 & 0 & 0 &
0 & 0 & 0 & 0 & 0 & 0 & 0 & 0 & \textbf{0} & 0 & 0 & 0 & 0 & 0 & 0 & 0
& 0 & 0 & 0 \\
\hline 15 & 0 & $\epsilon$ & 0 & 0 & 0 & 3$\epsilon$ &
0 & 0 & 0 & 5$\epsilon$ & 0 & 0 & 0 &7$\epsilon$& \textbf{0} & 0 & 0 &
9$\epsilon$ & 0 & 0 & 0 & 11$\epsilon$ & 0 & 0 \\
\hline 16 & 0 & 0 &
2 & 0 & 0 & 2 & 0 & 0 & 0 & 0 & 0 & 8 & 0 & 0 & 5 & \textbf{0} & 0 & 0
& 0 & 0 & 14 & 0 & 0 & 8 \\
\hline 17 & 0 & $\epsilon$ & 0 & 0 & 0 &
3$\epsilon$ & 0 & 0 & 0 & 5$\epsilon$ & 0 & 0 & 0 &7$\epsilon$& 0 & 0
& \textbf{0} & 9$\epsilon$ & 0 & 0 & 0 & 11$\epsilon$ & 0 & 0 \\
\hline 18 &0 & 0 & 0 & 0 & 0 & 0 & 0 & 0 & 0 & 0 & 0 & 0 & 0 & 0 & 0 &
0 & 0 & \textbf{0} & 0 & 0 & 0 & 0 & 0 & 0\\
\hline 19 & 0 &
$\epsilon$ & 0 & 0 & 0 & 3$\epsilon$ & 0 & 0 & 0 & 5$\epsilon$ & 0 & 0
& 0 &7$\epsilon$& 0 & 0 & 0 & 9$\epsilon$ & \textbf{0} & 0 & 0 &
11$\epsilon$ & 0 & 0 \\
\hline 20 & 0 & 0 & 0 & 0 & 0 & 0 & 0 & 0 & 0
& 0 & 0 & 0 & 0 & 0 & 0 & 0 & 0 & 0 & 0 & \textbf{0} & 0 & 0 & 0 & 0
\\
\hline 21 & 0 & $\epsilon$ & 0 & 0 & 0 & 3$\epsilon$ & 0 & 0 & 0 &
5$\epsilon$ & 0 & 0 & 0 &7$\epsilon$& 0 & 0 & 0 & 9$\epsilon$ & 0 & 0
& \textbf{0} & 11$\epsilon$ & 0 & 0 \\
\hline 22 & 0 & 0 & 0 & 0 & 0 &
0 & 0 & 0 & 0 & 0 & 0 & 0 & 0 & 0 & 0 & 0 & 0 & 0 & 0 & 0 & 0 &
\textbf{0} & 0 & 0 \\
\hline 23 & 0 & $\epsilon$ & 0 & 0 & 0 &
3$\epsilon$ & 0 & 0 & 0 & 5$\epsilon$ & 0 & 0 & 0 &7$\epsilon$& 0 & 0
& 0 & 9$\epsilon$ & 0 & 0 & 0 & 11$\epsilon$ & \textbf{0} & 0 \\
\hline 24 & 0 & 0 & 2 & 0 & 0 & 2 & 0 & 0 & 0 & 0 & 0 & 8 & 0 & 0 & 5
& 0 & 0 & 0 & 0 & 0 & 14 & 0 & 0 & \textbf{8} \\
\hline
\end{tabular}
\end{table}

A cursory look at Table \ref{TAB} supports the idea that when
$\textrm{Im}(\mathbf{G}_k(n))\not\equiv 0\pmod{n}$ (i.e., when
$\mathbf{G}_k(n)$ is not real modulo $n$) then
$\textrm{Re}(\mathbf{G}_k(n))\equiv \textrm{Im}(\mathbf{G}_k(n))\equiv n/2
\pmod{n}$. The large proportion of pairs $(k,n)$
for which $n \mid \mathbf{G}_k(n)$ is also remarkable.

The main goal of this paper is to give an analogue of Carlitz-von
Staudt formula in this Gaussian setting. In particular we prove
the following result:

\begin{teorpre}
  Let $k,n\geq 1$ be integers and consider the set
  $$
  \mathcal{P}(k,n):=\{ \textrm{prime $p$} : p \mid \mid n, p^2-1\mid k,
  p\equiv 3\pmod{4}\}.
  $$
Then:
$$
\mathbf{G}_k(n)\equiv
\begin{cases}  \frac{n}{2}(1+i) \pmod{n},
  & \textrm{if $k>1$ is odd and $n\equiv 2\pmod{4}$};\\
  \displaystyle{-\sum_{p\in\mathcal{P}(k,n)} \frac{n^2}{p^2}
    \pmod{n}}, & \textrm{otherwise}.
\end{cases}
$$
\end{teorpre}

As an application of this result, we study
the asymptotic density of the set of integers $n$ such that $n\mid \mathbf{G}_n(n)$,
(i.e., the density of zeros in the diagonal of Table \ref{TAB}). We prove
that this set has indeed an asymptotic density and compute its value
up to the sixth decimal digit
$0.971000\dots$.
This value is in contrast with that of the classical
integral setting \cite{GMO}, where the asymptotic density of the set
of integers $n$ such that $n\mid S_n(n)$ is exactly $1/2$.

\section{Auxiliary results on sums of binomial coefficients}\label{SEC:AUX}

In order to prove our main theorem we use some
technical results involving sums of binomial coefficients. The first
one is due to Hermite \cite{HER}, although Bachman \cite{BAC} gave it
in a more general form:

\begin{lem}\label{LEM:HER}
  Let $k$ be a positive integer and $p$ be a prime. Then:
  $$
  \sum_{0<j(p-1)<k} \binom{k}{j(p-1)}\equiv 0 \pmod{p}.
  $$
\end{lem}

The second technical result we use is more recent and is due to
Dilcher \cite{DIL}. It
involves alternating lacunary sums of binomial coefficients:

\begin{lem}\label{LEM:DIL}
  Let $k$ be a positive integer and let $p$ be an odd
  prime. Then
  $$
  \sum_{j=0}^{k}(-1)^{j}\binom{k(p-1)}{j(p-1)}\equiv
  \begin{cases}
    0\pmod{p}, & \textrm{if $k$ is odd};\\
    2\pmod{p}, & \textrm{if $k$ is even and $p+1\nmid k$};\\
    1 \pmod{p}, & \textrm{if $p+1\mid k$}.
  \end{cases}
  $$
\end{lem}

The following proposition will also play a key role in the proof of our
main theorem. It is a direct consequence of the lemmata above.

\begin{prop}\label{PROP:TEC}
  Let $p$ be an odd prime and $n$ a positive integer such that
  $p-1\mid n$.
  Then:
  $$
  \sum_{j=1}^{\frac{n}{(p-1)}-1}
  (-1)^{\frac{j(p-1)}{2}}\binom{n}{j(p-1)}\equiv\begin{cases}
    -1\pmod{p}, & \textrm{if $p\equiv 3\pmod{4}$ and $p+1\mid
      \frac{n}{(p-1)}$};\\ 0\pmod{p}, & otherwise.\end{cases}
  $$
\end{prop}
\begin{proof}
  Write $n=k(p-1)$. then the sum in the statement is

  $$
  S=\sum_{j=1}^{k-1}(-1)^{\frac{j(p-1)}{2}}\binom{k(p-1)}{j(p-1)}.
  $$

If $p \equiv 1 \pmod 4$, then $\frac{p-1}{2}$ is even and the sum $S$
does not alternate so that Lemma \ref{LEM:HER} applies and
$S\equiv 0\pmod{p}$ in this case.

On the other hand, if $p\equiv 3\pmod{4}$ $S$ alternates then

$$
S=\sum_{j=0}^{k}(-1)^{j}\binom{k(p-1)}{j(p-1)}-\binom{k(p-1)}{0}-(-1)^k\binom{k(p-1)}{k(p-1)}
$$

and the result follows by Lemma \ref{LEM:DIL}.
\end{proof}

\section{Proof of the main theorem}

Recall that
$$
\mathbf{G}_k(n):=
\sum_{1\leq a,b\leq
  n}(a+bi)^k.
$$

Writing $z=a+bi$, the binomial theorem gives:

$$
\mathbf{G}_k(n)\equiv\sum_{1\leq a\leq n}\sum_{1\leq b\leq n}\sum_{1\leq
  m\leq k}\binom{k}{m}a^{k-m}b^mi^m \pmod{n}.
$$

Consequently, from the definition of the power sum $S_k(n)$ we
obtain the following:

\begin{lem}\label{LEM:SUM}
  Let $k,n$ be positive integers. Then:
\begin{itemize}
\item[i)] $\displaystyle{\textrm{Re}(\mathbf{G}_k(n))\equiv
\sum_{j=0}^{\lfloor k/2\rfloor}
(-1)^j\binom{k}{2j}S_{2j}(n)S_{k-2j}(n)\pmod{n}.}$
\item[ii)] $\displaystyle{\textrm{Im}(\mathbf{G}_k(n))\equiv
\sum_{j=0}^{\lfloor (k-1)/2\rfloor}
(-1)^j\binom{k}{2j+1}S_{2j+1}(n)S_{k-2j-1}(n)\pmod{n}.}$
\end{itemize}
\end{lem}

This result allows us to study $\textrm{Re}(\mathbf{G}_k(n))$ and
$\textrm{Im}(\mathbf{G}_k(n))$ separately. We start with the
imaginary part:

\begin{prop}\label{PROP:IM}
  For any integers $n,k$, $\textrm{Im}(\mathbf{G}_k(n))\equiv 0\pmod n$ unless
  $n\equiv 2\pmod{4}$ and $k>1$ is odd in which case
  $\textrm{Im}(\mathbf{G}_k(n))\equiv n/2\pmod{n}$.
\end{prop}
\begin{proof} We examine different cases and use Lemma
\ref{LEM:SUM} ii) extensively.
\begin{itemize}
\item If $n$ is odd, then $p-1$ is even for every $p\mid n$ and we can
  apply part i) of Theorem \ref{TEO:GMO} to get $S_{2j+1}(n)\equiv
  0\pmod n$ for every $0\leq j\leq \lfloor k/2\rfloor+1$. Hence
  $\textrm{Im}(\mathbf{G}_k(n))\equiv 0\pmod n$ in this case.
\item If $4\mid n$ then Theorem \ref{TEO:GMO} ii) implies that
  $S_{2j+1}(n)\equiv 0\pmod n$ for every $j>1$. Consequently,
  $\textrm{Im}(\mathbf{G}_k(n))\equiv k\frac{n(n-1)}{2}S_{k-1}(n)\pmod n$
  and four cases arise:
  \begin{itemize}
  \item[i)] If $k=1$, then $\textrm{Im}(\mathbf{G}_k(n))\equiv
    n\frac{n(n-1)}{2}\equiv 0\pmod{n}$.
  \item[ii)] If $k=2$, then $\textrm{Im}(\mathbf{G}_k(n))\equiv
    2\left(\frac{n(n-1)}{2}\right)^2\equiv 0\pmod{n}$.
  \item[iii)] If $k>2$ is even, then $S_{k-1}(n)\equiv 0\pmod{n}$ due to
    Theorem \ref{TEO:GMO} ii) and hence $\textrm{Im}(\mathbf{G}_k(n))\equiv
    0\pmod{n}$.
  \item[iv)] If $k>1$ is odd, then \cite[Lemma 3]{GMO}
    $S_{k-1}(n)\equiv\frac{n}{2}S_{k-1}(2)\equiv 0\pmod{2}$ so that
    $\textrm{Im}(\mathbf{G}_k(n))\equiv 0\pmod{n}$.
\end{itemize}
\item For $n\equiv 2\pmod{4}$ we consider the following cases:
\begin{itemize}
\item[i)] If $k=1$, since $S_0(n)=n$ then
  $\textrm{Im}(\mathbf{G}_k(n))\equiv 0\pmod{n}$ trivially.
\item[ii)] If $k$ is even, then $\binom{k}{2j+1}$ is also even for every
  $j\geq0$. Moreover, we know \cite{GMO} that in this case
  $S_{2j+1}(n)\equiv 0\pmod{n/2}$ from which follows that
  $\textrm{Im}(\mathbf{G}_k(n))\equiv 0\pmod{n}$.
\item[iii)] If $k>1$ is odd, it is easy to see that
  $\displaystyle{\sum_{j=0}^{\frac{k-1}{2}}\binom{k}{2j+1}\equiv
  1\pmod{2}}$. Thus, since $S_m(2)\equiv 1\pmod{2}$ for every positive
  $m$, it follows that $\textrm{Im}(\mathbf{G}_k(n))\equiv 1\pmod{2}$.
  Just like in the previous case, $\textrm{Im}(\mathbf{G}_k(n))\equiv
  0\pmod{n/2}$ and then
  $\textrm{Im}(\mathbf{G}_k(n))\equiv n/2\pmod{n}$.
\end{itemize}
\end{itemize}
\end{proof}

We now consider the real part, which requires a finer analysis. Notice
that $S_k(1)=1\equiv 0\pmod{1}$, so that in what follows we assume
$n>1$.

\begin{prop}\label{PROP:RE1}
  If $n\equiv 2\pmod{4}$ and $k>1$ is odd, then
  $\textrm{Re}(\mathbf{G}_k(n))\equiv n/2\pmod{n}$.
\end{prop}
\begin{proof} Since $k$ is odd, $k-2j$ is odd for every $0\leq
j\leq\lfloor k/2\rfloor$. Consequently \cite{GMO} $S_{k-2j}(n)\equiv
0\pmod{n/2}$ and due to Lemma \ref{LEM:SUM} i),
$\textrm{Re}(\mathbf{G}_k(n))\equiv 0 \pmod{n/2}$.  Moreover, since
$S_0(2)\equiv 0\pmod{2}$ and $S_m(2)\equiv 1\pmod{2}$ for every $m>1$,
it follows that $\textrm{Re}(\mathbf{G}_k(n))\equiv
\dfrac{n^2}{4}\displaystyle{\sum_{j=1}^{\frac{k-1}{2}}\binom{k}{2j}}\pmod{2}$.
To conclude, it is enough to observe that
$\displaystyle{\sum_{j=1}^{\frac{k-1}{2}}\binom{k}{2j}}\equiv
1\pmod{2}$ and $n^2/4\equiv 1\pmod{2}$.
\end{proof}

\begin{prop}\label{PROP:RE2}
  If $k>1$ is odd and $n\not\equiv 2\pmod{4}$, or if
  $k=1$, then $\textrm{Re}(\mathbf{G}_k(n))\equiv 0\pmod{n}$.
\end{prop}
\begin{proof}
  The case $k=1$ is trivial, since
  $\textrm{Re}(\mathbf{G}_k(n))\equiv S_0(n)S_1(n)=nS_1(n)\equiv 0\pmod{n}$.

  Now, assume that $k>1$ is odd and $n\not\equiv 2\pmod{4}$. We
  distinguish two cases:
  \begin{itemize}
    \item[i)] If $n$ is odd, then $S_{k-2j}(n)\equiv0\pmod{n}$ for every
    $0\leq j\leq\lfloor k/2\rfloor$ because $k-2j$ is odd and Theorem
    \ref{TEO:GMO} i) applies. The result follows from Lemma \ref{LEM:SUM}
    i).
    \item[ii)] If $4\mid n$, then $S_{k-2j}(n)\equiv 0\pmod{n}$
    for every $0\leq j<\lfloor k/2\rfloor=\frac{k-1}{2}$. Hence,
    $\textrm{Re}(\mathbf{G}_k(n))\equiv
    (-1)^{\frac{k-1}{2}}\binom{k}{k-1}S_0(n)S_1(n)\equiv 0\pmod{n}$
    because $S_0(n)=n$.
  \end{itemize}
\end{proof}

\begin{prop}\label{PROP:RE3}
  Let $k>1$ and
  $n=p_1^{r_1}\cdots p_s^{r_s}$ be integers. Then
  $$
  \textrm{Re}(\mathbf{G}_k(n))\equiv
  \begin{cases}
    -\frac{n^2}{p_i^2} \pmod{p_i}, &
    \textrm{if $r_i=1$, $p_i^2-1\mid k$ and $p_i \equiv 3\pmod{4}$};\\
    0 \pmod{p_i^{r_i}}, & otherwise.
  \end{cases}
  $$
\end{prop}
\begin{proof}
  Since $k$ is even and $S_0(n)=0$, by Lemma
  \ref{LEM:SUM} i), for every $1\leq i\leq s$:
  $$
  \textrm{Re}(\mathbf{G}_k(n))\equiv \sum_{j=1}^{\frac{k}{2}-1}
  (-1)^j\binom{k}{2j}S_{2j}(n)S_{k-2j}(n)\pmod{p_i^{r_i}}.
  $$
As usual in this section, we study different cases:
\begin{itemize}
\item If $p_i=2$, as $2j$ and $k-2j$ are even for every $j$ we have
that \cite{GMO} $S_{2j}(n)\equiv
\frac{n}{2^{r_i}}S_{2j}(2^{r_i})\equiv\frac{n}{2^{r_i}}2^{r_i-1}\equiv
n/2\pmod {2^{r_i}}$ and, in the same way $S_{k-2j}(n)\equiv
n/2\pmod{2^{r_i}}$. Hence, $\textrm{Re}(\mathbf{G}_k(n))\equiv
\frac{n^2}{4}\sum_{j=1}^{\frac{k}{2}-1}
(-1)^j\binom{k}{2j}\pmod{2^{r_i}}$. Now:
\begin{itemize}
  \item[i)] If $r_i>1$, clearly $n^2/4\equiv 0\pmod{2^{r_i}}$ because
  $2r_i-2\geq r_i$ and thus $\textrm{Re}(\mathbf{G}_k(n))\equiv
  0\pmod{2^{r_i}}$.
  \item[ii)] If $r_i=1$, we have that $\sum_{j=1}^{\frac{k}{2}-1}
  (-1)^j\binom{k}{2j}\equiv \sum_{j=1}^{\frac{k}{2}-1}
  \binom{k}{2j}\equiv 0\pmod{2}$ so, again,
  $\textrm{Re}(\mathbf{G}_k(n))\equiv 0\pmod{2^{r_i}}$.
\end{itemize}
\item If $p_i$ is an odd prime, then \cite{GMO}:
  $$
  S_m(p_i^{r_i})\equiv
  \begin{cases} -p_r^{r_i-1} \pmod{p_i^{r_i}}, &
    \textrm{if $p_i-1\mid m$};\\
    0 \pmod {p_i^{r_i}}, & \textrm{otherwise}.
  \end{cases}
  $$
Which gives:
\begin{itemize}
\item[i)] If $r_i>1$, then every term in the expression of
$\textrm{Re}(\mathbf{G}_k(n))$ is $0\pmod{p_i^{r_i}}$:
\begin{itemize}
  \item[a)] If either $p_i-1\nmid 2j$ or $p_i-1\nmid k-2j$, then either
  $S_2j(n)\equiv 0\pmod{p_i^{r_i}}$ or $S_{k-2j}(n)\equiv 0
  \pmod{p_i^{r_i}}$.
  \item[b)] If $p_i-1$ divides both $2j$ and $k-2j$, then
  $S_{2j}(n)S_{k-2j}(n)\equiv p^{2r_i-2}\equiv 0\pmod{p_i^{r_i}}$.
\end{itemize}
\item[ii)] If $r_i=1$ and $p_i-1\nmid k$, then, for every $1\leq j\leq
k/2-1$, either $p_i-1\nmid 2j$ or $p_i-1\nmid k-2j$. Thus, every
term in the expression of $\textrm{Re}(\mathbf{G}_k(n))$ is
$0\pmod{p_i}$.
\item[iii)] If $r_i=1$ and $p_i-1\mid k$, then, for every $1\leq j\leq
k/2-1$ either $p_i-1\mid 2j$ or $p_i-1\nmid 2j$. If $p_i-1\nmid 2j$, then
the corresponding term is $0\pmod{p_i}$. If $p_i-1\mid 2j$, so that
$p_i-1\mid k-2j$ and thus $S_{2j}(n)\equiv S_{k-2j}(n)\equiv
n/p_i\pmod{p_i}$. Consequently,
$$
\textrm{Re}(\mathbf{G}_k(n))\equiv\frac{n^2}{p_i^2}
\sum_{\substack{1\leq j\leq k/2-1\\ p_i-1\mid 2j}}
(-1)^j\binom{k}{2j}=\sum_{j=1}^{\frac{k}{p_i-1}-1}(-1)^{\frac{j(p_i-1)}{2}}\binom{k}{j(p_i-1)}.
$$
But this latter sum can be evaluated using Proposition \ref{PROP:TEC}
to complete the proof in this case.
\end{itemize}
\end{itemize}
\end{proof}

\begin{teor}
\label{TEOR:MAIN} Let $k,n\geq 1$ be integers. Define the set
$$
\mathcal{P}(k,n):=\{ \textrm{prime $p$} : p \mid \mid n, p^2-1\mid k,
p\equiv 3\pmod{4}\}.
$$
Then:
$$
\mathbf{G}_k(n)\equiv
\begin{cases}
  \frac{n}{2}(1+i) \pmod{n},
    & \textrm{if $k>1$ is odd and $n\equiv 2\pmod{4}$};\\
  \displaystyle{-\sum_{p\in\mathcal{P}(k,n)} \frac{n^2}{p^2}
    \pmod{n}},
    & \textrm{otherwise}.
\end{cases}
$$
\end{teor}
\begin{proof}
From Propositions \ref{PROP:IM} and \ref{PROP:RE1}, we
know that $\mathbf{G}_k(n)\equiv\frac{n}{2}(1+i)\pmod{n}$ if $k>1$ is odd
and $n\equiv 2\pmod{4}$.

In the remaining cases, $\textrm{Im}(\mathbf{G}_k(n))\equiv 0\pmod{n}$ by
Proposition \ref{PROP:IM}.

Define
$n'=\displaystyle{\prod_{p\in\mathcal{P}(k,n)}p}$. Clearly
$n=\frac{n}{n'}\cdot n'$ and $\gcd(n/n',n')=1$. Propositions
\ref{PROP:RE2} and \ref{PROP:RE3} imply that:
$$
\textrm{Re}(\mathbf{G}_k(n))\equiv 0\pmod{n/n'},
$$
$$
\textrm{Re}(\mathbf{G}_k(n))\equiv -n^2/p^2\pmod{p},\
\textrm{for every $p\in\mathcal{P}(k,n)$}.
$$
And the result follows applying the Chinese Remainder Theorem.
\end{proof}

\section{On the congruence $\mathbf{G}_k(n) \equiv 0 \pmod{n}$}

In this section we focus on the solutions to $\mathbf{G}_k(n) \equiv 0 \pmod{n}$; i.e. numbers $n$ such that $0=\displaystyle{\sum_{z\in\mathbb{Z}_n[i]} z^k }$
. In particular, we study the sets
$$
\mathcal{N}_k:=\{n \in \mathbb{N} : \mathbf{G}_k(n) \equiv 0 \pmod{n} \},
$$
$$
\mathcal{K}_n:=\{k \in \mathbb{N} : \mathbf{G}_k(n) \equiv 0 \pmod{n} \}.
$$
In other words, we are interested in the zeros of each row and
column in Table \ref{TAB}.

The following result is a simple consequence of Theorem
\ref{TEOR:MAIN}:

\begin{cor}
\label{COR:NK} Let $k,n\geq 1$ be integers. Then $\mathbf{G}_k(n) \not
\equiv 0\pmod{n}$ if and only if there exists a prime $p$ dividing $n$
such that:
\begin{itemize}
\item[i)] $p \equiv 3 \pmod{4}$.
\item[ii)] $p^2-1 \mid k$.
\item[iii)] $p^2 \nmid n$. %
\end{itemize}
\end{cor}

This corollary will allow us to explicitly describe the complements of
$\mathcal{N}_k$ and $\mathcal{K}_n$ and, furthermore, to obtain
information about their density.

\begin{prop}
\label{PROP:N} Let $p$ be a prime a define the set $\mathbb{F}(p) :=
\{p (p s+r): \textrm{$s \in \mathbb{N}$ and $0<r<p$} \}$. Then:
$$
\mathbb{N} \setminus\mathcal{N}_k=
\begin{cases}
  4\mathbb{N}+2, &
    \textrm{if $k>1$ is odd}; \\
  \displaystyle{\bigcup_{\substack{p^2-1\mid k\\p\equiv 3\pmod{4}}}
    \mathbb{F}(p)}, & \textrm{otherwise.}
\end{cases}
$$
\end{prop}

\begin{prop}\label{PROP:K} 
  Let $p$ be a prime and define the set $\mathbb{G}(p) :=
\{h(p^2-1): h \in \mathbb{N} \}$. Then:
$$\mathbb{N} \setminus\mathcal{K}_n=\begin{cases} 2\mathbb{N}+1, & \textrm{if $n \equiv 2\pmod{4}$}; \\ \displaystyle{\bigcup_{\substack{p \|  n \\ p \equiv 3 \pmod 4}} \mathbb{G}(p)}, & \textrm{otherwise}. \end{cases}$$
\end{prop}

In what follows, given a set $A\subseteq\mathbb{N}$, we denote by
$\delta(A)$ its asymptotic density.

\begin{teor}
  For very positive integer $k$, the asymptotic density of
$\mathcal{N}_k$ is:
$$
\delta(\mathcal{N}_k)=
\begin{cases}
  3/4, & \textrm{if $k>1$ is odd}; \\
  \displaystyle{\prod_{\substack{p^2-1 \mid k\\ p \equiv 3
        \pmod{4}}}}\frac{p^2-p+1}{p^2}, &
    \textrm{otherwise}.
\end{cases}
$$
\end{teor}
\begin{proof}
  For any non-empty finite 
  family of primes $\mathcal{P}$, the
system of congruences
$$
\{x \equiv pr \pmod {p^2} : p\in\mathcal{P}\}
$$
has solutions. An easy inductive argument shows that
$$
\delta\left(\bigcap_{p\in
    \mathcal{P}}\mathbb{F}(p)\right)=\displaystyle{
  \frac{\displaystyle{\prod_{p \in \mathcal{P}}p-1}}{{\rm lcm} \{p^2:
    p \in \mathcal{P}\}}}=\prod_{p \in \mathcal{P} }\frac{p-1}{p^2}.
$$
Proposition \ref{PROP:N} and the inclusion-exclusion principle
lead to
$$
\delta(\mathbb{N}\setminus\mathcal{N}_k)=1-\prod_{\substack{p^2-1 \mid
    k\\p \equiv 3 \pmod{4}}} (1-\frac{p-1}{p^2})
$$
and we are done.
\end{proof}

This result has the following somewhat remarkable consequence:

\begin{cor}\label{COR:SOR}
  For every $\epsilon>0$, there exists $k \in
\mathbb{N}$ such that $
\delta(\mathcal{N}_k)<\epsilon$.
\end{cor}
\begin{proof} It is enough to observe that
  $$
  \prod_{\substack{p \equiv 3 \pmod 4}}\frac{p^2-p+1}{p^2}=0.
  $$
\end{proof}

\begin{rem}
  Corollary \ref{COR:SOR} means that, despite great
  amount of zeros in Table \ref{TAB}, there are rows such that the
  density of zeros on them is as close to 0 as desired.
\end{rem}

\begin{prop}\label{PROP:8}
  Let $n$ be a positive integer. If $3\mid n$ but $9\nmid
n$, then $8\mathbb{N}\subseteq \mathbb{N}\setminus \mathcal{K}_n$. If,
in addition, $n\not\equiv 2\pmod{4}$, then $8\mathbb{N}=
\mathbb{N}\setminus \mathcal{K}_n$.
\end{prop}
\begin{proof}
  If $8\mid k$, then $k\in\mathbb{G}(3)$. Hence, if $3\mid
n$ and $9\nmid n$, Proposition \ref{PROP:K} implies that
$k\in\mathbb{N}\setminus\mathcal{K}_n$.

I we furthermore assume that $n\not\equiv 2\pmod{4}$, then Proposition
\ref{PROP:K} implies that, if $k\in\mathbb{N}\setminus\mathcal{K}_n$,
then $p^2-1\mid k$ for some $p\mid n$ such that $p\equiv
3\pmod{4}$. But in this case, $p^2-1\equiv 0\pmod{8}$ and the proof
is complete.
\end{proof}

\section{On the congruence $\mathbf{G}_n(n) \equiv 0 \pmod{n} \}$}

We consider in this section the case $k=n$; i.e., we are concerned
with
those $n$ such that $n\mid \mathbf{G}_n(n)$. In other words: the
zeros in the diagonal of Table \ref{TAB}.

The following result is just a version of Corollary \ref{COR:NK} when
$k=n$.

\begin{cor}\label{COR:N}
  Let $n>1$ be an integer. Then, $\mathbf{G}_n(n) \not \equiv
  0\pmod{n}$ if and only if there exists
  a prime $p$ such that:
\begin{itemize}
\item[i)] $p\equiv 3 \pmod{4}$.
\item[ii)] $p^3-p \mid n$.
\item[iii)] $p^2 \nmid n$.
\end{itemize}
\end{cor}

As a consequence we obtain a result similar to Proposition
\ref{PROP:8}:

\begin{prop}\label{PROP:24}
  Let $n$ be a positive integer. If $\mathbf{G}_n(n)\not
\equiv 0 \pmod{n}$, then $24 \mid n$.
\end{prop}
\begin{proof}
  By Corollary \ref{COR:N}, if $\mathbf{G}_n(n) \not
\equiv 0 \pmod{n}$ then $n=hp(p+1)(p-1)$ for some prime $p \equiv 3
\pmod{4}$, so that $8 \mid (p+1)(p-1)$. Moreover, one of $-1$,
$p$ or $p+1$ is a multiple of 3 and we are done.
\end{proof}

Define the following set:
$$
\mathfrak{M}:=\{n \in \mathbb{N} :\mathbf{G}_n(n) \equiv 0 \pmod n\}
$$

The rest of the paper is devoted to computing the asymptotic density
of $\mathfrak{M}$. Note that Proposition \ref{PROP:24} implies that
this density (if it exists) is, at least,
$\frac{23}{24}=0.958\overline{3}$. In fact we show that it is
quite close to this value computing $\delta(\mathfrak{M})$ up to five
decimal places.

For a prime $p$, define the following set:
$$\mathfrak{U}_p:=\{n \in\mathbb{Z} : p^3-p \mid n, p^2 \nmid n\}.$$

\begin{prop}
  The set $\mathfrak{M}$ satisfies the following
conditions:
\begin{itemize}
\item[i)] $\displaystyle{\mathbb{N}\setminus\mathfrak{M}
=\bigcup_{\substack{p\ \textrm{prime}\\ p\equiv 3\pmod{4}}}
\mathfrak{U}_p}$.
\item[ii)] $\mathfrak{M}$ has an asymptotic density.
\end{itemize}
\end{prop}
\begin{proof}
The first assertion is a straightforward consequence of Corollary
\ref{COR:N}.\\
In order to prove ii), let $u_p=p^3-p$ and observe that
$u_p=\min(\mathfrak{U}_p)$. Then, $\mathfrak{U}_p = u_p\mathbb{Z}
\setminus pu_p\mathbb{Z}$ and, consequently,
$$
\delta(\mathfrak{U}_p)=\frac{1}{u_p}-\frac{1}{pu_p}=
\frac{1}{p^2(1+p)}.
$$
Since $\sum_{p}\delta(\mathfrak{U}_p)<\infty$, it follows that
$\mathbb{N}\setminus\mathfrak{M}$ has an asymptotic density and
so has $\mathfrak{M}$, as claimed.
\end{proof}

In order to compute bounds for the asymptotic density of
$\mathfrak{M}$ (now we know it exists) we present a couple of
technical lemmata.

\begin{lem}
\label{LEM:EQ} Let $2<q<p$ be two prime numbers and $0<s<p$, $0<t<q$
two integers. The Diophantic equation
\begin{equation*} (p^3-p)(Yp+s) = (q^3-q)(Xq+t)
\end{equation*} has a solution if and only if $q^2\nmid p^2-1$.
\end{lem}
\begin{proof} Rewriting the equality as
  \begin{equation*} Kp^2(p^2-1) + sp(p^2-1) = K^{\prime}q^2(q^2-1) +
tq(q^2-1)
  \end{equation*} and taking the $\gcd$:
  \begin{equation*} p^2-1 = m\bar{p},\,\,\,q^2-1 = m\bar{q}
  \end{equation*} the original equation simplifies to
  \begin{equation*} Kp^2\bar{p} + sp\bar{p} =
K^{\prime}q^2\bar{q}+tq\bar{q}.
  \end{equation*} There are three cases to consider, depending on
$\gcd(\bar{p},q^2)$ (notice that $2<q$ implies $p\nmid \bar{q}$
because $p\nmid (q+1)(q-1)$).
  \begin{itemize}
  \item If $\gcd(\bar{p}, q^2)=1$ then the same happens with
$p^2\bar{p}$ and $q^2\bar{q}$, so that the equality is of the form
    \begin{equation}\label{eq:siempre-sol} Kp_1 = K^{\prime}p_2+b
    \end{equation} for $p_1$ and $p_2$ coprime, which has an infinite
number of solutions for any $b$.
  \item If $\gcd(\bar{p}, q^{2})=q$ then one can divide by $q$ both
sides of the equation to get
    \begin{equation*} Kp^2\tilde{p} + sp\tilde{p} = K^{\prime}q\bar{q}
+ t\bar{q},
    \end{equation*} with, again, $p^2\tilde{p}$ and $q\bar{q}$ coprime
and we have another equation like \eqref{eq:siempre-sol}.
  \item Finally, if $\gcd(\bar{p}, q^2)=q^2$ then, dividing both sides
by $q$ the equation becomes
    \begin{equation*} Kp^2q\tilde{p} + spq\tilde{p} =
K^{\prime}q\bar{q} + t\bar{q},
    \end{equation*} which has no solutions because $t<q$.
  \end{itemize}
\end{proof}

\begin{lem}\label{LEM:SU}
  For $p, s$ integers, define
$\mathfrak{F}(p,s):=\{p(p-1)(p+1) (Kp+s): K \in \mathbb{N}\}$. If
$\mathcal{P}$ is a finite family of primes and $\{s_q\}_{q \in\mathcal{P}}$
satisfies $0<s_q<q$, then:
$$
\delta\left(\bigcap_{q\in \mathcal{P}} \mathfrak{F}(q,s_q)\right)=
\begin{cases} 0,
    & \textrm{if there exist $p,q \in \mathcal{P}$ with
    $p^2 \mid  q^2-1$};  \\
  \displaystyle{\frac{1}{{\rm lcm} \{q^4-q^2: q \in \mathcal{P}\}}},
    & otherwise. \end{cases}
$$
\end{lem}
\begin{proof}
  If there exist $p,q \in \mathcal{P}$ with $p^2 \mid
q^2-1$, Lemma \ref{LEM:EQ} implies that
$\mathfrak{F}(p,s_p)\bigcap\mathfrak{F}(q,s_q) = \emptyset$ and hence
$\bigcap_{q\in \mathcal{P}} \mathfrak{F}(q,s_q)= \emptyset$.

In the other case, by the Chinese Remainder Theorem, the set of solutions
of the system of simultaneous congruences given by:
$$
\{ x \equiv s_p (p^3-p) \pmod {p^4-p^2} : p \in \mathcal{P}\}
$$
determines an arithmetic progression of difference ${\rm lcm}
\{p^4-p^2: p \in \mathcal{P}\}$. Consequently its asymptotic density
is ${1/\rm lcm} \{q^4-q^2: q \in \mathcal{P}\}$ as claimed.
\end{proof}

We return to the sets $\mathfrak{U}_p$ previously defined.

\begin{prop}\label{PROP:DENU}
  Let $\mathcal{P}$ be a finite family of
primes. Then:
$$
\delta\left(\bigcap_{q\in \mathcal{P}}
  \mathfrak{U}_q\right)=
\begin{cases}
  0, &\textrm{if there are $p,q
    \in \mathcal{P}$ with $p^2 \mid  q^2-1$}; \\
  \displaystyle{\frac{\displaystyle{\prod_{q \in
          \mathcal{P}}q-1}}{{\rm lcm} \{q^4-q^2: q \in
      \mathcal{P}\}}}, & otherwise.
\end{cases}
$$
\end{prop}
\begin{proof} By induction on the number of elements in $\mathcal{P}$
and using Lemma \ref{LEM:SU} it can be shown that
 the intersection $\bigcap_{q\in \mathcal{P}}
 \mathfrak{U}_q$ (when non-empty) is the union of
 ${\prod_{q \in \mathcal{P}}q-1}$ disjoint
arithmetic progressions of difference $ {\rm lcm} \{q^4-q^2: q \in
\mathcal{P}\}$.
\end{proof}

If $w(m)$ denotes the number of different prime factors of $m$,
 $\phi$ is the Euler totient function and defining
 $$
 \vartheta(m):=\begin{cases} 0, &\textrm{if there exist $p ,q \mid m$
     such that $p^2 \mid  q^2-1$}; \\
   \displaystyle{\frac{\phi(m)}{{\rm lcm} \{p^4-p^2: p  \mid m\}}}, &
   otherwise,\end{cases}
 $$
 then, the inclusion-exclusion principle together with the last
 Proposition let us state the following result:

 \begin{prop}
   Let $\mathcal{P}$ be a finite set of Gaussian primes and
$\Theta := \prod_{p \in P}$, then:

$$
\delta\left(\bigcup_{p\in \mathcal{P}}
  \mathfrak{U}_p\right)=-\sum_{1<d \mid \Theta} (-1)^{w(d)}
\vartheta(d)
$$
\end{prop}

These results allow us to approximate the
asymptotic density of $\mathfrak{M}$ which is given by the following sum:

$$
\delta(\mathfrak{M})=\sum_{m \in \Upsilon} {(-1)^{w(m)} \vartheta(m) }
$$
where $\Upsilon$ is the set of square-free integers whose prime
factors are all Gaussian.

\begin{teor}
  The asymptotic density of $\mathfrak{M}$ is
$0.971000\dots$
\end{teor}

\begin{proof}
  Let $\mathcal{P}$ be the set of the first thirty
Gaussian primes. Namely,
$$\mathcal{P}:=\{p\textrm{ prime}: p \equiv 3 \pmod 4  , p\leq 263\}.$$
Then:
$$
\delta\left(\bigcup_{p\in \mathcal{P}} \mathfrak{U}_p\right)\leq
\delta(\mathbb{N}\setminus\mathfrak{M})\leq
\delta\left(\bigcup_{p\in\mathcal{P}}
  \mathfrak{U}_p\right)+\sum_{\substack{p>263 \\ p
    \equiv 3 \pmod{4}}} \frac{1}{p^3+p^2}.
$$
Applying the inclusion-exclusion principle, and taking into
account Proposition \ref{PROP:DENU}, we have been able to compute,
using PARI/GP:
$$
\ell:=\delta\left(\bigcup_{p\in \mathcal{P}} \mathfrak{U}_p\right)=
\frac{52832172344...086951451}{1821843350513...659697280} = 0.0289992947691577872...
$$
where the numerator has $117$ digits and the denominator has $119$.
We know (see A085992 in the OEIS or \cite{COH}) that
$$ \sum_{\substack{p\ \textrm{prime} \\ p
    \equiv 3 \pmod{4}}} \frac{1}{p^3}
 =: \Theta= 0.0410075565664730319288865488519600259243
\dots
$$
Moreover, if $\mathfrak{p}:=1299689$ is the $99999$-th
prime, then one can compute

$$
\sum_{\substack{\mathfrak{p}<p\ \textrm{prime} \\ p
    \equiv 3 \pmod{4}}} \frac{1}{p^2+p^3}<\sum_{\substack{\mathfrak{p}< p\ \textrm{prime} \\ p
    \equiv 3 \pmod{4}}} \frac{1}{p^3}=\sum_{\substack{\mathfrak{p}\geq p\ \textrm{prime} \\ p
    \equiv 3 \pmod{4}}} \frac{-1}{p^3}+\Theta <2 \times 10^{-14},
$$
$$\sum_{\substack{263<p\leq\mathfrak{p}\\ p \equiv 3 \pmod{4}}} \frac{1}{p^3+p^2}<5.3539\times 10^{-7},$$
Consequently:
$$
0.0289992947<\ell<\delta(\mathbb{N}\setminus\mathfrak{M})<\ell+5.354\times
10^{-7}<0.0289998302,
$$
and hence:
$$0.971000169 < \delta(\mathfrak{M})<0.97100071.$$
\end{proof}

\begin{rem}
  The computation of the asymptotic density of $\mathfrak{M}$ up to
  $6$ decimal digits has required over $24$ hours. Albeit the
  implementation does not use either parallelism or caching, the fact
  that the computational complexity of the problem is essentially
  $\mathcal{O}(2^n)$ (due to the underlying inclusion-exclusion
  principle), trying to get to the $57$ Gaussian primes required for the next
  decimal digit has been seen by us as not not worth the effort, as we
  do not have access either to massively parallel hardware or large
  amounts of RAM.
\end{rem}

\section{Conclusions and future perspectives}
We have started with this work an interesting new research area on the
sum of powers on the ring
$\mathbb{Z}[i]/n\mathbb{Z}[i]$.
The formulas in Theorem \ref{TEOR:MAIN} allow a fast computation of
that sum from the Gaussian prime factors of $n$, in an analogue way as
von Staudt's formula for
$\mathbb{Z}_n$. There are also two areas of interest that this work
opens before us:
\subsection{Sums of powers in more general rings.}
A more general framework might be described as follows: given a finite
ring $\mathcal{A}$, 
find a formula for the value
of $\sum_{a \in \mathcal{A}} a^k$. Natural first steps might
$\mathcal{A}$ being the ring of square matrices of a given order with
coefficients in $\mathbb{Z}_n$ or the ring of Hamilton quaternions over
$\mathbb{Z}_n$, $\mathbb{H}(\mathbb{Z}_n)$. However, these cases might prove too complicated due
to their non-commutativity and the lack of results similar to those of
Section \ref{SEC:AUX}. At the same time, conjectures are not easy to
come up with, as computations soon become unfeasible for $n$ a little
large. As a matter of fact, we have found no pair $(k,n)$ such that
the sum of the $k-$th powers of the elements of
$\mathbb{H}(\mathbb{Z}_n)$ be nonzero.

On the other hand, the  numbers $n$ (up to $n=246$) for which the sum of of the $n$-th
powers of all $2 \times 2$ matrices over $\mathbb{Z}/n\mathbb{Z}$ is
non-zero are shown in  the OEIS sequence A236810. All of them are of
congruent with $ 6 \pmod {12} $, but this is not something we would
conjecture as a fact for all $ n \in \mathbb {N} $.

\subsection{The {\em Erd\H{o}s-Moser equation} in Gaussian stage}
We would like to finish this paper posing in the Gaussian context a topic
related to power sums of integers as the {\em
Erd\H{o}s-Moser equation}, which is the Diophantine equation
\begin{equation}
\label{eq} S_k(m-1)=m^k.
\end{equation}
In a 1950 letter to Moser, Erd\H{o}s conjectured that
solutions to this equation do not exist except for the trivial
one $1^1+2^1=3^1$. Three years later, Moser \cite{MOS} proved the
conjecture for odd $k$ or $m < 10^{10^6}$. Since then, much work on
this equation has been carried out, but the conjecture has not
been completely solved. For surveys of research on this and related
problems, see \cite{BUT,MOR} and \cite[Section D7]{Guy}.

For power sums of Gaussian integers, a reasonable analogue
Diophantine equation could be
$$
\mathbf{G}_k(m-1)=(m+m i)^k
$$
for which, after performing computations for
$k,m<100$, we state the following

\begin{con}
  The equation above has only the solution $(k,m)=(2,3)$:
$$ (1 + i)^2 + (1 + 2i)^2 + (2 + i)^2 + (2 + 2i)^2 = 18 i = (3 + 3i)^2$$
\end{con}

\end{document}